\documentclass[11pt]{amsart}

\usepackage{amsmath}
\usepackage{amssymb}
\usepackage{amsfonts}
\usepackage{mathrsfs}
\usepackage{bm}

\usepackage{graphicx,epsfig}
\usepackage[all]{xy}
\usepackage{tikz}

\DeclareMathAlphabet{\mathcal}{OMS}{cmsy}{m}{n}

\theoremstyle{plain}
\newtheorem{theorem}{Theorem}[section]
\newtheorem{lemma}[theorem]{Lemma}
\newtheorem{proposition}[theorem]{Proposition}
\newtheorem{corollary}[theorem]{Corollary}

\newtheorem*{maintheorem}{Theorem \ref{thm:catenoid charaterization}}

\theoremstyle{definition}

\theoremstyle{remark}
\newtheorem{remark}[theorem]{Remark}

\numberwithin{equation}{section}

\allowdisplaybreaks

\makeatletter
\def\@citestyle{\m@th\upshape\mdseries}
\def\citeform#1{{\color{blue}\bfseries#1}}
\def\@cite#1#2{{%
  \@citestyle[\citeform{#1}\if@tempswa, #2\fi]}}
\@ifundefined{cite }{%
  \expandafter\let\csname cite \endcsname\cite
  \edef\cite{\@nx\protect\@xp\@nx\csname cite \endcsname}%
}{}
\makeatother

\makeatletter
\newcommand*\rel@kern[1]{\kern#1\dimexpr\macc@kerna}
\newcommand*\widebar[1]{%
  \begingroup
  \def\mathaccent##1##2{%
    \rel@kern{0.8}%
    \overline{\rel@kern{-0.8}\macc@nucleus\rel@kern{0.2}}%
    \rel@kern{-0.2}%
  }%
  \macc@depth\@ne
  \let\math@bgroup\@empty \let\math@egroup\macc@set@skewchar
  \mathsurround\z@ \frozen@everymath{\mathgroup\macc@group\relax}%
  \macc@set@skewchar\relax
  \let\mathaccentV\macc@nested@a
  \macc@nested@a\relax111{#1}%
  \endgroup
}
\makeatother

\renewcommand{\leq}{\leqslant}
\renewcommand{\geq}{\geqslant}

\newcommand{\ddl}[2]{\dfrac{d{#1}}{d{#2}}}
\newcommand{\ddz}[2]{\dfrac{d^2{#1}}{d{#2}^2}}

\newcommand{\R}{\mathbb{R}}

\newcommand{\B}{\mathbb{B}}
\renewcommand{\H}{\mathbb{H}}

\newcommand{\I}{\mathbf{I}}
\renewcommand{\SS}{\mathbb{S}}

\DeclareMathOperator{\diag}{diag}%
\DeclareMathOperator{\SO}{\mathsf{SO}}%

\newcommand{\CC}{\mathcal{C}}
\renewcommand{\O}{\mathsf{O}}%

\newcommand{\Mscr}{\mathscr{M}}

\newcommand{\Rbar}{\widebar{R}}
\newcommand{\Mbar}{\widebar{M}} 

\newcommand{\gtilde}{\tilde{g}}
\newcommand{\xtilde}{\tilde{x}}

\title[Simons' equation and minimal hypersurfaces]
{Simons' equation and minimal hypersurfaces in space forms}

\begin{document}

\author{Biao Wang}
\thanks{This research was partially supported by
PSC-CUNY Research Award \#{}68119-0046.}


\date{\today}

\address{Department of Mathematics and Computer Science\\
         The City University of New York, QCC\\
         222-05 56th Avenue, Bayside, NY 11364\\}
\email{biwang@qcc.cuny.edu}

\begin{abstract}
  Let $n\geq{}3$ be an integer, and let $\Sigma^n$ be a non totally geodesic
  complete minimal hypersurface immersed in the
  $(n+1)$-dimensional space form $\Mbar^{n+1}(c)$, where the constant
  $c$ denotes the sectional curvature of the space form. If $\Sigma^n$
  satisfies the Simons' equation \eqref{eq:Simons equation}, then either
  (1) $\Sigma^n$ is a catenoid if $c\leq{}0$, or (2) $\Sigma^n$ is a
  Clifford minimal hypersurface or a compact Ostuki minimal hypersurface
  if $c>0$. This paper is motivated by Tam and Zhou \cite{TZ09}.
\end{abstract}

\maketitle

\section{Introduction}

In 1968 Simons \cite{Sim68} (see also \cite[$\S$7 or $\S$9]{Che68} and
\cite[$\S$1.6]{Xin03}) showed that the second fundamental form of an immersed
minimal hypersurface in the sphere or in the Euclidean space
satisfies a second order elliptic partial differential equation,
which can imply the famous Simons' inequality.
The Simons' inequality enabled him to prove a gap phenomenon for minimal submanifolds
in the sphere and the Bernstein's problem in $\R^n$ for $n\leq{}7$.
Since then the Simons' inequality has been used by various authors to study minimal
immersions.

In this paper, we shall classify all complete minimal hypersurfaces
immersed in the space form $\Mbar^{n+1}(c)$ which satisfy the Simons' equation
\eqref{eq:Simons equation}, where $n\geq{}3$.

Roughly speaking, a \emph{catenoid} is a minimal
rotation hypersurface immersed in $\Mbar^{n+1}(c)$.
In the case when $c=0$, Tam and Zhou \cite[Theorem 3.1]{TZ09} proved that if
a non-flat complete minimal hypersurface $\Sigma^n$ immersed in $\R^{n+1}$ satisfies
the Simons' equation \eqref{eq:Simons equation} on all nonvanishing points
of $|A|$, then $\Sigma^n$ must be a catenoid.

Motivated by the ideas of Tam and Zhou, we generalize the result to the cases
when $c\ne{}0$. More precisely we will prove the following theorem.

\begin{theorem}
\label{thm:catenoid charaterization}
Let $\Mbar^{n+1}(c)$ be the space form of dimension $n+1$, where
$n\geq{}3$. Suppose that $\Sigma^n$ is a
non totally geodesic complete minimal hypersurface immersed in $\Mbar^{n+1}(c)$.
If the Simons' equation \eqref{eq:Simons equation} holds as an equation at
all nonvanishing points of $|A|$ in $\Sigma^n$,
then $\Sigma^n\subset\Mbar^{n+1}(c)$ is either
\begin{enumerate}
  \item a catenoid if $c\leq{}0$, or
  \item a Clifford minimal hypersurface or a compact Ostuki
        minimal hypersurface if $c>0$.
\end{enumerate}
\end{theorem}

\begin{remark}
The compact minimal rotation hypersurfaces in
Theorem \ref{thm:otsuki} are called the \emph{Otsuki minimal hypersurfaces}
(see \cite{Ots70}). The Otsuki minimal hypersurfaces are \emph{immersed}
catenoids.
\end{remark}

\begin{remark}
Combing Proposition \ref{prop:Clifford_Simons_equation}
and Proposition \ref{prop:catenoid_Simons identity} with
Theorem \ref{thm:catenoid charaterization},
we can see that the Clifford minimal hypersurfaces and the
catenoids are the \emph{only} complete minimal hypersurfaces
satisfying the Simons' equation \eqref{eq:Simons equation}.
\end{remark}

\begin{remark}We say that $\Sigma^n$ is a \emph{complete} minimal hypersurface
in the space form $\Mbar^{n+1}(c)$, we actually mean one of the following cases:
\begin{enumerate}
  \item if $c\leq{}0$, then $\Sigma^n$ is a noncompact hypersurface
        without boundary, that is, an open hypersurface, or
  \item if $c>0$, then $\Sigma^n$ is a compact hypersurface without
        boundary, that is, a closed hypersurface.
\end{enumerate}
\end{remark}

\noindent\textbf{Plan of the paper.}
This paper is organized as follows:
In $\S$\ref{sec:prelim} we define the catenoids and their generating curves
in the space forms.
In $\S$\ref{sec:Simons_equation} we derive the Simons' identity
\eqref{eq:Simons identity}, and we show that the Clifford minimal hypersurfaces
and the catenoids satisfy \eqref{eq:Simons equation}.
In $\S$\ref{sec:Proof_Main_Theorem}
we prove Theorem \ref{thm:catenoid charaterization}.
In $\S$\ref{sec:appendix} we offer some figures of the generating curves
of the catenoids in the space forms.

\section{Preliminary}\label{sec:prelim}

A simply connected $(n+1)$-dimensional complete Riemannian manifold
whose sectional curvature is equal to a constant $c$,
denoted by $\Mbar^{n+1}(c)$, is called a \emph{space form}.
There are three types of space forms:

(i) If $c>0$, let
\begin{equation}
   \Mbar^{n+1}(c)=\SS^{n+1}(c)=\{x\in\R^{n+2}\ |\
   x_{1}^{2}+\cdots+x_{n+2}^{2}=1/c\}\ .
\end{equation}

(ii) If $c<0$, let
\begin{equation}
   \Mbar^{n+1}(c)=\B^{n+1}(c)=
   \{x\in\R^{n+1}\ |\ x_{1}^{2}+\cdots+x_{n+1}^{2}<-1/c\}\ ,
\end{equation}
with the metric
\begin{equation}
   ds^{2}=\frac{4|dx|^{2}}{(1+c|x|^2)^{2}}\ ,
\end{equation}
where $x=(x_{1},\ldots,x_{n+1})$ and $|x|^{2}=x_{1}^{2}+\cdots+x_{n+1}^{2}$.

(iii) If $c=0$, let
\begin{equation}
   \Mbar^{n+1}(0)=\R^{n+1}
\end{equation}
be the $(n+1)$-dimensional Euclidean space.

In Theorem \ref{Thm:Clifford_minimal_hypersurface} and Theorem \ref{thm:otsuki},
the results that we quote are about minimal hypersurfaces immersed in the
unit sphere $\SS^{n+1}$, but these results can be generalized to the
space forms $\SS^{n+1}(c)$ for $c>0$.

In fact, consider both $\SS^{n+1}$ and $\SS^{n+1}(c)$ ($c>0$) as the subsets of
$\R^{n+2}$. Define the map $f:\R^{n+2}\to\R^{n+2}$
by $f(x)={x}/{\sqrt{c}}$ for any $x\in\R^{n+2}$, where $c>0$. It's easy to verify
$\SS^{n+1}(c)=f(\SS^{n+1})$. For any hypersurface $\Sigma^{n}$
immersed in $\SS^{n+1}$, let $\Sigma^{n}(c)=f(\Sigma^{n})$, then $\Sigma^{n}(c)$ is
a hypersurface immersed in $\SS^{n+1}(c)$.

\begin{lemma}\label{lem:conformal_minimal}
If $\Sigma^{n}$ is a minimal hypersurface immersed in $\SS^{n+1}$,
then $\Sigma^{n}(c)$ is a minimal hypersurface immersed in $\SS^{n+1}(c)$.
\end{lemma}

\begin{proof}Let $g_{ij}$ and $\gtilde_{ij}$ be the first fundamental forms
of the hypersurfaces $\Sigma^{n}\subset\SS^{n+1}$ and
$\Sigma^{n}(c)\subset\SS^{n+1}(c)$ respectively, then
$\gtilde_{ij}=g_{ij}/c$ and $\gtilde^{ij}=cg^{ij}$ for $1\leq{}i,j\leq{}n$.
It's easy to verify that the Laplacians on $\Sigma^{n}$ and $\Sigma^{n}(c)$
satisfy the equation $\Delta_{\Sigma^{n}(c)}=c\Delta_{\Sigma^{n}}$.
Let $x$ and $\xtilde$ be the position functions of $\Sigma^{n}$ and
$\Sigma^{n}(c)$ in $\R^{n+2}$ respectively, then $\xtilde=x/\sqrt{c}$.

If $\Sigma$ is a minimal hypersurface immersed in $\SS^{n+1}$,
then by Theorem 3 in \cite{Tak66} (see also \cite[p.101]{dCW70} or
\cite[Theorem 3.10.2]{Ji04})
we have $\Delta_{\Sigma^{n}}x=-nx$. On the other hand,
we have the following identities
\begin{equation*}
   \Delta_{\Sigma^{n}(c)}\xtilde=c\Delta_{\Sigma^{n}}\left(\frac{x}{\sqrt{c}}\right)
   =\sqrt{c}\,\Delta_{\Sigma^{n}}x=\sqrt{c}\,(-nx)=-nc\xtilde\ ,
\end{equation*}
which can imply that $\Sigma^{n}(c)$ is a minimal hypersurface immersed
in $\SS^{n+1}(c)$ by applying \cite[Theorem 3]{Tak66} again.
\end{proof}

As we will see that any complete minimal hypersurface
$\Sigma^n$ immersed in $\Mbar^{n+1}(c)$ satisfying the equation
\eqref{eq:Simons equation} at all nonvanishing points of $|A|$
is a minimal rotation hypersurface unless $\Sigma^n$ is a Clifford minimal
hypersurface in the case when $c>0$, so at first we shall study the
Clifford minimal hypersurfaces and minimal rotation hypersurfaces in the
space forms.

\subsection{Clifford minimal hypersurfaces in $\SS^{n+1}(c)$}
\label{subsec:Clifford_torus}

In this subsection, we shall define the Clifford minimal hypersurfaces
in the space forms $\SS^{n+1}(c)$ for $c>0$.
Let $S^q(r)$ be a $q$-dimensional sphere in $\R^{q+1}$ with
radius $r$. In particular $\SS^{q}(c)=S^{q}(1/\sqrt{c})$ for $c>0$.

For $c>0$ and $m=1,\ldots,n-1$,
a \emph{Clifford minimal hypersurface} embedded in $\SS^{n+1}(c)$ is
defined as follows
\begin{equation}\label{eq:Clifford hypersurfaces}
   \Mscr_{m,n-m}(c)=S^{m}\left(\sqrt{\frac{m}{cn}}\right)\times
   S^{n-m}\left(\sqrt{\frac{n-m}{cn}}\right)\ .
\end{equation}
In particular, $\Mscr_{m,n-m}=\Mscr_{m,n-m}(1)$ is a Clifford minimal
hypersurface embedded in $\SS^{n+1}$ (see also
\cite{Che68} and \cite[pp.229--230]{Ji04}).
The following result is well known.



\begin{theorem}[{\cite{CdCK70,Law69}}]
\label{Thm:Clifford_minimal_hypersurface}
The Clifford minimal hypersurfaces $\Mscr_{m,n-m}$ are the only compact
minimal hypersurfaces in $\SS^{n+1}$ with $|A|^2=n$.

Furthermore the second fundamental form $A$ has two distinct constant
eigenvalues with multiplicities $m$ and $n$ respectively.
\end{theorem}

\subsection{Catenoids in space forms}
\label{subsec:rotation_hypersurface} 
In this subsection we shall follow Hsiang \cite{Hsi82,Hsi83} to derive the
differential equations of the generating curves of catenoids in the space form
$\Mbar^{n+1}(c)$, and solve the differential equations in the case when $c\leq{}0$.

Let $G=\SO(n)$ be a subgroup of the orientation preserving
isometry group of $\Mbar^{n+1}(c)$ which pointwise fixes a
given geodesic $M^1\subset{}\Mbar^{n+1}(c)$.
We call $G$ the \emph{spherical group} of $\Mbar^{n+1}(c)$ and
$M^1$ the \emph{rotation axis} of $G$. A hypersurface in
$\Mbar^{n+1}(c)$ that is invariant under $G$ is called a
\emph{rotation hypersurface}; 
if the rotation hypersurface in $\Mbar^{n+1}(c)$ is a complete
minimal hypersurface, then it is called a \emph{spherical catenoid} or just
\emph{catenoid}, denoted by $\CC$.

The following $2$-dimensional half space is well defined
\begin{equation}\label{eq:half space}
   M_{+}^2(c)=\Mbar^{n+1}(c)/G \ .
\end{equation}
There are three types of the half spaces:
\begin{enumerate}
  \item $\SS_{+}^{2}(c)=\{x_{1}^{2}+x_{n+1}^{2}+x_{n+2}^{2}=1/c\ |\
         x_{1}\geq{}0\}=M_{+}^{2}(c)$ if $c>0$.
  \item $\B_{+}^{2}(c)=\{x_{1}^{2}+x_{n+1}^{2}<-1/c\ |\ x_{1}\geq{}0\}=M_{+}^{2}(c)$
        if $c<0$.
  \item $\R_{+}^{2}=\{(x_{1},x_{n+1})\in\R^2\ |\ x_1\geq{}0\}=M_{+}^{2}(0)$ if
        $c=0$.
\end{enumerate}
It's easy to see that the rotation axis $M^1$ is the boundary of $M_{+}^2(c)$.

The orbital distance metric on $\Mbar^{n+1}(c)/G$ is the same
as the restriction metric of $M_{+}^2(c)$. Let $d(\cdot,\cdot)$
be the distance function defined on $M_{+}^2(c)$.
We shall parametrize $M_{+}^2(c)$ by the following coordinate
system: Choose a base point $O\in{}M^1$ and let $x$ be the arc
length on $M^1$ travelling in the positive orientation of
$M^1= \partial{}M_{+}^2(c)$. To each point $p\in{}M_{+}^2(c)$,
there is a (unique) point $q\in{}M^1$ such that the length of
the geodesic arc connecting $p$ and $q$, denoted by
$\overline{pq}$, is equal to $d(p,M^1)$. We shall assign to the
point $O$ the coordinate $(x,y)$, where $x=d(O,q)$ and
$y=d(p,q)=$ the length of the geodesic arc $\overline{pq}$
(see Figure \ref{fig:warped product metric}).

\begin{figure}[htbp]
\begin{center}
  \begin{minipage}{.45\textwidth}
  \centering
  \includegraphics[scale=0.8]{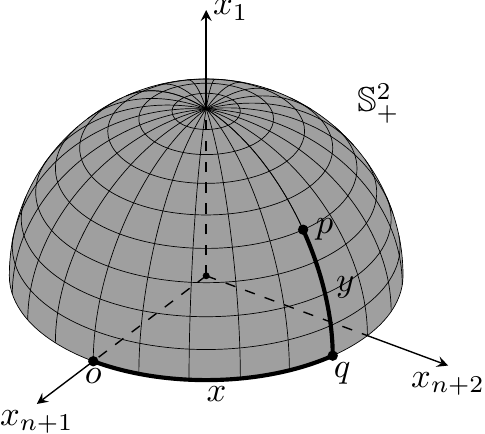}
  \end{minipage}%
  \begin{minipage}{.55\textwidth}
  \centering
  \includegraphics[scale=0.85]{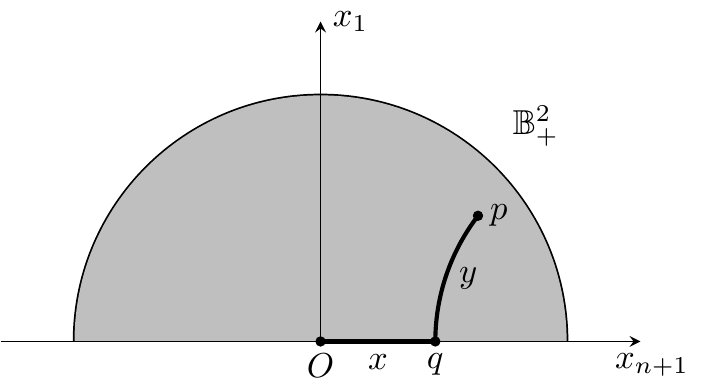}
  \end{minipage}
\end{center}
\caption{The warped product metric on the half space
    $M_{+}^2(c)$ for $c=\pm{}1$. In each half space,
    $x=d(O,q)$ and $y=d(p,q)=$ the length of the geodesic
    arc $\overline{pq}$.}\label{fig:warped product metric}
\end{figure}

According to the above definition of $x$ and $y$, we have
\begin{equation*}
\begin{cases}
     -\infty<x<\infty\ \text{and}\
     0\leq{}y<\infty\ ,
   & \text{if}\ c\leq{}0\ , \\
     -\dfrac{\pi}{\sqrt{c}}\leq{}x<\dfrac{\pi}{\sqrt{c}}
     \ \text{and}\ 0\leq{}y\leq\dfrac{\pi}{2\sqrt{c}}\ ,
   & \text{if}\ c>0\ .
\end{cases}
\end{equation*}
In the case $c>0$, the coordinate of the center is
$\left(x,\dfrac{\pi}{2\sqrt{c}}\right)$, where $x$ is arbitrary.
The \emph{warped product metric} on $M_{+}^2(c)$ is
written in the form
\begin{equation}\label{eq:warped product metric}
   ds^2=(f'(y))^{2}\cdot{}dx^2+dy^2\ ,
\end{equation}
where $f'=df/dy$ and
\begin{equation}
   f(y)=
   \begin{cases}
        \dfrac{1}{\sqrt{-c}}\,\sinh(\sqrt{-c}\,y)\ ,
      & \text{if}\ c<0\ (\text{hyperbolic case})\ ,\\
        y\ ,
      & \text{if}\ c=0\ (\text{Euclidean case})\ ,\\
        \dfrac{1}{\sqrt{c}}\,\sin(\sqrt{c}\,y)\ ,
      & \text{if}\ c>0\ (\text{spherical case})\ .
   \end{cases}
\end{equation}

Let $\Sigma^n$ be a rotation hypersurface in $\Mbar^{n+1}(c)$
with respect to the geodesic $M^1$, then the curve
$\gamma=\Sigma^n\cap{}M_{+}^2(c)$ is called the
\emph{generating curve} of $\Sigma^n$.
Suppose that $\gamma$ is given by the parametric equations:
$x = x(s)$ and $y = y(s)$, $a\leq{}s\leq{}b$, where
$s$ is the arc length of $\gamma$ and $y(s)>0$. Let
$\alpha$ be the angle between the unit tangent vector
of $\gamma(s)$ and $\partial{}/\partial{}y$ (see Figure
\ref{fig:diff eq of gamma}).

\begin{figure}[htbp]
  \begin{center}
     \includegraphics[scale=0.8]{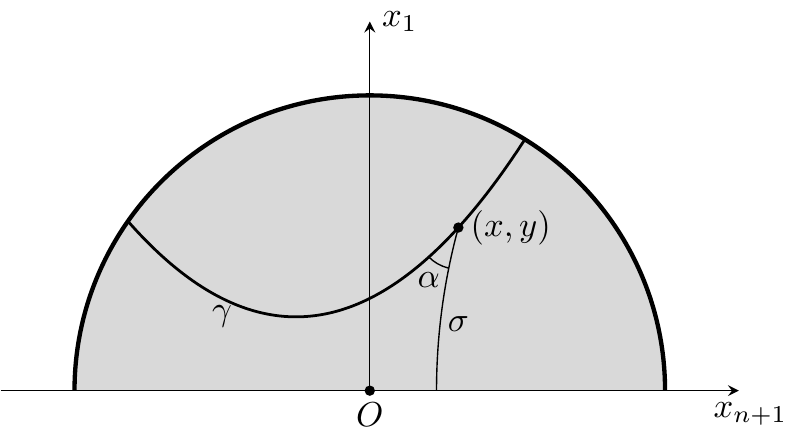}
   \end{center}
  \caption{In the hyperbolic half space $\B_{+}^2$,
  $\alpha$ is the angle between the parametrized
  curve $\gamma$ and the geodesic $\sigma$ at the point
  $(x,y)$, where $\sigma$ is
  perpendicular to the $x_{n+1}$-axis.}\label{fig:diff eq of gamma}
\end{figure}

Now suppose that the mean curvature of the rotation hypersurface
$\Sigma^n\subset\Mbar^{n+1}(c)$ generated by the curve $\gamma\subset{}M_{+}^2(c)$
is zero, then we get the following differential equations of $\gamma$
(see \cite[pp. 487--488]{Hsi82} for the details)
\begin{equation}\label{eq:diff equation of gamma}
   \frac{f^{n-1}\cdot{}(f')^2}{\sqrt{(f')^2+(y')^2}}=
   f^{n-1}\cdot{}f'\cdot{}\sin\alpha=k\ (\text{constant})\ ,
\end{equation}
where $y'=dy/dx$.

Without loss of generality we may consider the differential equations
\eqref{eq:diff equation of gamma} with initial data
$y(0)=a>0$ and $y'(0)=0$. Plugging the initial conditions
into \eqref{eq:diff equation of gamma}, we get
$k = f^{n-1}(a)f'(a)$, which implies the following equation
\begin{equation}\label{eq:bounded condition}
   \sin\alpha=\frac{f^{n-1}(a)f'(a)}{f^{n-1}(y)f'(y)}\ .
\end{equation}
Now we assume $c\leq{}0$, then we can write $dx/dy$ as follows
\begin{equation*}
   \ddl{x}{y}=\frac{1}{f'(y)}\cdot
        \frac{dy}{\sqrt{\left(\dfrac{f(y)}{f(a)}\right)^{2n-2}
        \cdot\left(\dfrac{f'(y)}{f'(a)}\right)^2-1}}\ .
\end{equation*}
Integrating both sides in terms of $y$, we have the following function
\begin{equation}\label{eq:generating_curve_function}
   x(y)=\int_{a}^{y}\frac{1}{f'(t)}\cdot
        \frac{dt}{\sqrt{\left(\dfrac{f(t)}{f(a)}\right)^{2n-2}
        \cdot\left(\dfrac{f'(t)}{f'(a)}\right)^2-1}}\ ,
\end{equation}
where $a\leq{}y<\infty$.


\subsection{Catenoids in $\SS^{n+1}(c)$}\label{subsec:Otsuki's curves}

In this subsection we follow Otsuki \cite{Ots70}
to study the generating curves of the compact immersed
minimal rotation hypersurfaces in $\SS^{n+1}$.
See also \cite{BL90} in \emph{equivariant language}.

\begin{figure}[htbp]
  \begin{center}
     \includegraphics[scale=0.8]{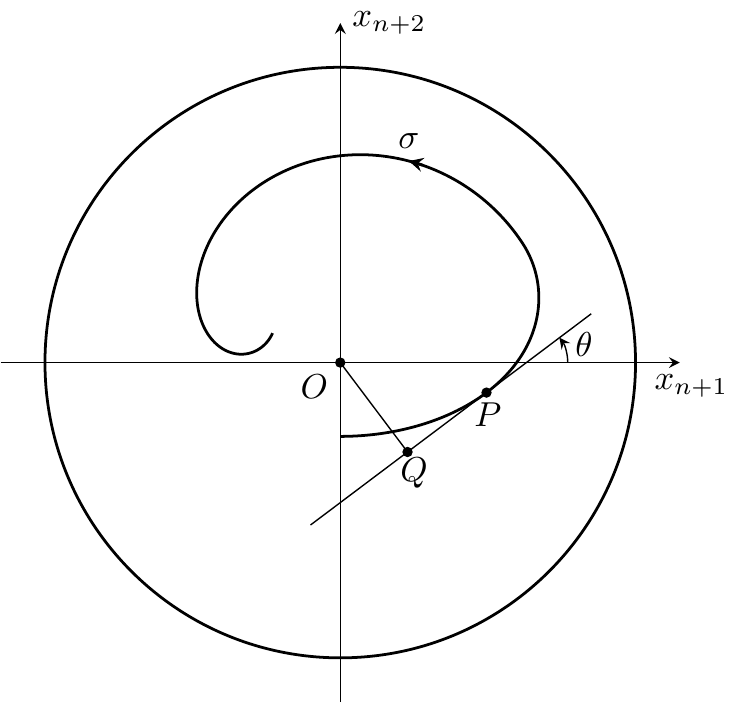}
   \end{center}
  \caption{The generating curve for the spherical catenoid in $\SS^{n+1}$,
  its support function is $h(\theta)=d(O,Q)$ and $h'(\theta)=d(P,Q)$.
  The coordinates $(x_{n+1},x_{n+2})$ of the point $P$ (inside the unit disk)
  are given by $x_{n+1}=h\sin\theta+h'\cos\theta$ and
  $x_{n+2}=-h\cos\theta+h'\sin\theta$ (see equation (4.2) in \cite{Ots70}).}
  \label{Fig: support function}
\end{figure}

Suppose that $\gamma\subset\SS_{+}^2$ is the generating curve of a minimal
rotation hypersurface in $\SS^{n+1}$. Project orthogonally the curve $\gamma$
into the $x_{n+1}x_{n+2}$-plane, and denote it by $\sigma$.
Let $h(\theta)$ is the support function of the curve $\sigma$,
Otsuki \cite{Ots70} proved that $h$ satisfies the differential equation
\begin{equation}\label{eq:ODE of the support function}
      nh(1-h^2)\ddz{h}{\theta}+\left(\ddl{h}{\theta}\right)^2+
      (1-h^2)(nh^2-1)=0
\end{equation}
with the initial conditions $h(0)=a\leq{1}/\sqrt{n}$
and $h'(0)=0$ (see Figure~\ref{Fig: support function}).

If $a=1/\sqrt{n}$, then the minimal rotation hypersurface in $\SS^{n+1}$
generated by the curve $\sigma$ satisfying \eqref{eq:ODE of the support function} with
this initial data is $\Mscr_{n-1,1}$, one of the Clifford minimal hypersurfaces
(see \cite[p.~160]{Ots70}).

From now on we may assume that $0<a<1/\sqrt{n}$, and set
\begin{equation}\label{eq:constant in ODE}
   C(a)=(a^2)^{1/n}(1-a^2)^{1-1/n}=a^{2/n}(1-a^2)^{1-1/n}
\end{equation}
for $0<a<1/\sqrt{n}$. Otsuki \cite{Ots70,Ots72} (see also \cite{LW07})
proved that the support
function $h(\theta)$, which is the solution to the differential equation
\eqref{eq:ODE of the support function} with the initial conditions
$h(0)=a\leq{1}/\sqrt{n}$ and $h'(0)=0$, is a periodic function whose
period is
\begin{equation}\label{eq:period of support function}
   T(a)=2\int_{a_0}^{a_1}\frac{dx}
   {\sqrt{1-x^2-C(a)\left(\dfrac{1}{x^2}-1\right)^{1/n}}}\ ,
\end{equation}
where $a_0=a\in(0,1/\sqrt{n})$, and $a_1\in(1/\sqrt{n},1)$  is a solution
to the equation
\begin{equation*}
   1-x^2-C(a)\left(\frac{1}{x^2}-1\right)^{1/n}=0\ .
\end{equation*}
It was proved in \cite{Ots70,Ots72,LW07} that the period $T$
satisfies the following conditions:
\begin{enumerate}
  \item $T(a)\in(\pi,2\pi)$ is differentiable on $(0,1/\sqrt{n})$,
  \item $\displaystyle\lim_{a\to{}0^{+}}T(a)=\pi$ and
        $\displaystyle\lim_{a\to\left(1/\sqrt{n}\right)^{-}}T(a)=\sqrt{2}\,\pi$.
\end{enumerate}
Moreover the generating curve $\sigma$ is a simple closed curve if and only if
$T(a)=2\pi/k$ for $k=1,2,\ldots$, and $\sigma$ is a closed curve (not necessarily
simple) if and only if $T(a)$ is a (positive) rational multiple of $\pi$.

In conclusion we have the following results:

\begin{theorem}[\cite{Ots70,Ots72,BL90,LW07}]\label{thm:otsuki}
Let $n\geq{}3$ be an integer.
\begin{enumerate}
  \item There is no closed minimally embedded rotation hypersurface of
        $\SS^{n+1}$ other than the Clifford minimal hypersurface
        $\Mscr_{n-1,1}$ and the round geodesic sphere $\SS^{n}$.
  \item There are countably infinitely many closed minimal rotation
        hypersurfaces immersed in $\SS^{n+1}$ {\rm(}see also \cite{Hsi83}{\rm)}.
\end{enumerate}
\end{theorem}


\section{Simons' equation and catenoids in space forms}\label{sec:Simons_equation}

Suppose that $\Sigma^n$ is a hypersurface immersed in the
$(n+1)$-dimensional space form $\Mbar^{n+1}(c)$.
Let $A$ be the second fundamental form of $\Sigma^n$ and $\nabla{}A$ be
the covariant derivative of $A$, and let $h_{ij}$ and $h_{ijk}$ be the
components of $A$ and $\nabla{}A$ in an orthonormal frame respectively.

The following lemma was proved by Tam and Zhou
\cite[Lemma 3.1]{TZ09} for the case when $c=0$, but it's also true
for the case when $c\ne{}0$.

\begin{lemma}Let $\Sigma^n$ be a minimal hypersurface
immersed in the space form $\Mbar^{n+1}(c)$.
At a point where $|A|>0$, we have
\begin{equation}\label{eq:Simons identity}
   |A|{}\Delta|A|+|A|^4=\frac{2}{n}\,|\nabla|A|{}|^2
   +nc|A|^2+E\ ,
\end{equation}
with $E\geq{}0$. Moreover, in an orthonormal frame such that
$h_{ij} = \lambda_{i}\delta_{ij}$, then $E = E_1 + E_2 + E_3$, where
\begin{equation}\label{eq:error term of Simons equation}
\begin{aligned}
  E_1 & = \sum_{j\ne{}i,k\ne{}i,k\ne{}j}h_{ijk}^2\ ,\\
  E_2 & = \frac{2}{n} \sum_{j\ne{}i,k\ne{}i,k\ne{}j}
          (h_{kki}-h_{jji})^2\ ,\\
  E_3 & = \left(1+\frac{2}{n}\right)|A|^{-2}\sum_{k}
          \sum_{i\ne{}j}(h_{ii}h_{jjk}-h_{jj}h_{iik})^2\ .
\end{aligned}
\end{equation}
\end{lemma}

\begin{proof}For any point $p\in{}\Sigma^n$, we choose an
orthonormal frame field $e_1,\ldots$, $e_{n+1}$ such that,
restricted to $\Sigma^n$, the vectors $e_1,\ldots,e_n$ are
tangent to $\Sigma^n$ and the vector $e_{n+1}$ is perpendicular
to $\Sigma^n$, and the second fundamental form of
$\Sigma^n$ is diagonalized by $h_{ij}=\lambda_{i}\delta_{ij}$,
where $1\leq{}i,j\leq{}n$.

Recall that the curvature tensor $\Rbar_{ABCD}$ of
$\Mbar^{n+1}(c)$ is given by
\begin{equation}\label{eq:space_form_curvature}
   \Rbar_{ABCD}=c(\delta_{AC}\delta_{BD}-\delta_{AD}\delta_{BC})\ ,
   \quad{}1\leq{}A,B,C,D\leq{}n+1\ ,
\end{equation}
where $\delta_{AB}$ is the Kronecker delta.
According to \cite[(3.1)]{CdCK70} and \cite[(1.21) and (1.27)]{SSY75}, we have
\begin{equation}
  \sum_{i,j}h_{ij}\Delta{}h_{ij}=-|A|^{4}+nc{}|A|^2\ ,
\end{equation}
and
\begin{equation}
  |A|\Delta{}|A|+|\nabla{}|A||^2=\frac{1}{2}\Delta|A|^2=
  \sum_{i,j,k}h_{ijk}^2+\sum_{i,j}h_{ij}\Delta{}h_{ij}\ ,
\end{equation}
where $h_{ijk}$ denotes the component of the covariant derivative
of the second fundamental form $A$ of $\Sigma^n$
for $1\leq{}i,j,k\leq{}n$. Therefore we have
\begin{equation}\label{eq:simons01}
  |A|\Delta{}|A|+|\nabla{}|A||^2=|\nabla{}A|^2-|A|^{4}+nc{}|A|^2\ ,
\end{equation}
where $|\nabla{}A|^2=\sum\limits_{i,j,k}h_{ijk}^2$.
We claim that
\begin{equation}\label{eq:simons02}
  |\nabla{}A|^2=\left(1+\frac{2}{n}\right)|\nabla{}|A||^2+E\ .
\end{equation}
In fact, according to the computation in
\cite[{pp. 3456--3457}]{TZ09}, we have
\begin{align*}
  |\nabla{}A|^2-|\nabla{}|A||^2
  & = E_1+2\sum_{i\ne{}k}h_{iik}^2+\frac{n}{n+2}\,E_3\\
  & = E_1+\frac{2}{n}\left(|\nabla{}|A||^2+\frac{n}{n+2}\,E_3+
      \frac{n}{2}\,E_2\right)+\frac{n}{n+2}\,E_3\\
  & = E_1+\frac{2}{n}|\nabla{}|A||^2+E_2+E_3\ ,
\end{align*}
where we use the fact $h_{ijk}=h_{ikj}$ since the sectional curvature of
$\Mbar^{n+1}(c)$ is constant
(see \cite[(2.12)]{CdCK70} or \cite[(1.10)]{SSY75}).

Combining \eqref{eq:simons01} and \eqref{eq:simons02} together, we have
\eqref{eq:Simons identity}.
\end{proof}

Obviously, the term $E$ in \eqref{eq:Simons identity} is always nonnegative,
which implies the famous Simons' inequality
\begin{equation}\label{eq:Simons_inequality}
   |A|{}\Delta|A|+|A|^4\geq\frac{2}{n}\,|\nabla|A|{}|^{2}+nc|A|^2\ .
\end{equation}
If $E\equiv{}0$ in \eqref{eq:Simons identity}, we get the
\emph{Simons' equation}
\begin{equation}\label{eq:Simons equation}
   |A|{}\Delta|A|+|A|^4=\frac{2}{n}\,|\nabla|A|{}|^2+nc|A|^2\ .
\end{equation}
If $n=2$, then $E\equiv{}0$ in \eqref{eq:Simons identity},
so we have the following corollary.

\begin{corollary}If $\Sigma^2$ is a minimal surface
immersed in the $3$-dimensional space form $\Mbar^{3}(c)$,
then $\Sigma^2$ satisfies the following Simons' equation
\begin{equation}\label{eq:Simons equation (n=2)}
   |A|{}\Delta|A|+|A|^4=|\nabla|A|{}|^2+2c|A|^2\ .
\end{equation}
\end{corollary}

We may ask what kinds of non totally geodesic minimal hypersurfaces in the
space form $\Mbar^{n+1}(c)$ satisfy the Simons' equation
\eqref{eq:Simons equation}?
Proposition \ref{prop:Clifford_Simons_equation} and
Proposition \ref{prop:catenoid_Simons identity} show that
the Clifford minimal hypersurfaces \eqref{eq:Clifford hypersurfaces}
and the minimal rotation hypersurfaces (i.e., the catenoids)
satisfy \eqref{eq:Simons equation}. 

On the other hand, Theorem \ref{thm:catenoid charaterization} shows that
the Clifford minimal hypersurfaces and catenoids are the only non totally
geodesic complete minimal hypersurfaces satisfying
\eqref{eq:Simons equation}. 

\begin{proposition}\label{prop:Clifford_Simons_equation}
When $c>0$, the second fundamental form $|A|$ of each
Clifford minimal hypersurface \eqref{eq:Clifford hypersurfaces}
in $\SS^{n+1}(c)$ satisfies \eqref{eq:Simons equation}.
\end{proposition}

\begin{proof}
We can prove the statement by direct computation (see
\cite[pp.68-70]{CdCK70} or \cite[pp.229--230]{Ji04} for the details).
Because of Lemma \ref{lem:conformal_minimal}, we will just prove the statement
for the case when $c=1$.

For $m=1,\ldots,n-1$, we may embed the Clifford minimal hypersurface
$\Mscr_{m,n-m}$ into $\SS^{n+1}$ as follows.
Let $(u, v)$ be a point of $\Mscr_{m,n-m}$ where $u$ is a vector in $\R^{m+1}$
of length $\sqrt{m/n}$, and $v$ is a vector in $\R^{n-m+1}$ of length
$\sqrt{(n-m)/n}$. We can consider $(u, v)$ as a vector in
$\R^{n+2} =\R^{m+1}\times\R^{n-m+1}$ of length $1$.
Then we may choose an orthonormal basis on $\Mscr_{m,n-m}$ such that
the second fundamental form of $\Mscr_{m,n-m}$ can be written as follows
\begin{equation*}
   h_{ij}=\diag\bigg(
   \underbrace{\sqrt{\frac{n-m}{n}},\ldots,\sqrt{\frac{n-m}{n}}}_{m},
   \underbrace{-\sqrt{\frac{m}{n-m}},\ldots,-\sqrt{\frac{m}{n-m}}}_{n-m}
   \bigg)\ .
\end{equation*}
Since each component of $h_{ij}$ is a constant, we have $E=0$ in
\eqref{eq:Simons identity}. On the other hand, it's easy to get
\begin{equation*}
   |A|^{2}=m\cdot\frac{n-m}{m}+(n-m)\cdot\frac{m}{n-m}=n\ ,
\end{equation*}
so the Clifford minimal hypersurfaces satisfy \eqref{eq:Simons equation}.
\end{proof}

Next we shall verify that any catenoid in the space form
$\Mbar^{n+1}(c)$ satisfies the Simons' equation
\eqref{eq:Simons equation}.
In the case when $c=0$, Proposition \ref{prop:catenoid_Simons identity}
was proved by Tam and Zhou \cite[Proposition 2.1 (iv)]{TZ09}.

\begin{proposition}\label{prop:catenoid_Simons identity}
The second fundamental form $|A|$ of each catenoid $\CC$ in the space
form $\Mbar^{n+1}(c)$
satisfies the Simons' equation \eqref{eq:Simons equation}.
\end{proposition}

\begin{proof}We shall prove the statement in the unified way by
using the argument in \cite[$\S$ 2 and $\S$3]{dCD83}.
Consider the space
form $\Mbar^{n+1}(c)$ as a subset of $\R^{n+2}$ as follows:

(i) If $c>0$, let
\begin{equation*}
   \SS^{n+1}(c)=\{x\in\R^{n+2}\ |\ g_{1}(x,x)=1/c\}=\Mbar^{n+1}(c)\ ,
\end{equation*}
where $g_{1}(x,y)=x_{1}y_{1}+\cdots+x_{n+1}y_{n+1}+x_{n+2}y_{n+2}$
for $x,y\in\R^{n+2}$.

(ii) If $c<0$, let
\begin{equation*}
   \H^{n+1}(c)=\{x\in\R^{n+2}\ |\ g_{-1}(x,x)=1/c, x_{n+2}>0\}
   =\Mbar^{n+1}(c)\ ,
\end{equation*}
where $g_{-1}(x,y)=x_{1}y_{1}+\cdots+x_{n+1}y_{n+1}-x_{n+2}y_{n+2}$
for $x,y\in\R^{n+2}$.

(iii) If $c=0$, let
\begin{equation*}
   \R^{n+1}=\{x\in\R^{n+2}\ |\ x_{n+2}=0\}=\Mbar^{n+1}(0)\ .
\end{equation*}

Let $e_{i}=(0,\cdots,0,\underset{i^{\rm{th}}}{1},0,\cdots,0)$ be
the $i$-th vector in the space $\R^{n+2}$ for $i=1,\ldots,n+2$.
Let $P^{2}$ be a subspace of $\R^{n+2}$ spanned by either $e_{n+1}$ and $e_{n+2}$
if $c\ne{}0$ or $e_{n+1}$ if $c=0$, and
let $\O(P^2)$ be the set of metric-preserving transformations of $(\R^{n+2},g_{1})$
if $c>0$, $(\R^{n+2},g_{-1})$ if $c<0$ or $\R^{n+1}$ if $c=0$,
which leave $P^2$ pointwise fixed.
Let $P^{3}$ be a subspace of $\R^{n+2}$ spanned by either $e_{1}$,
$e_{n+1}$ and $e_{n+2}$ if $c\ne{}0$ or $e_{1}$ and $e_{n+1}$ if $c=0$.

Let $M^2(c)=\Mbar^{n+1}(c)\cap{}P^{3}$, and
let $\gamma$ be a smooth curve in $M^2(c)$ that does not meet
$P^2$. The orbit of $\gamma$ under the action of $\O(P^2)$ is a
rotation hypersurface generated by $\gamma$,
and the curve $\gamma$ is the generating curve of the
rotation hypersurface.

Suppose that the generating curve $\gamma$ is parametrized by either
$x_{1}=x_{1}(s)$, $x_{n+1}=x_{n+1}(s)$ and $x_{n+2}=x_{n+2}(s)$ if $c\ne{}0$ or
$x_{1}=x_{1}(s)$ and $x_{n+1}=x_{n+1}(s)$ if $c=0$, where $s$
is the arc length parameter of the curve $\gamma$.
Let $\I$ be either a straight line if $c\leq{}0$ or
a closed curve immersed in a plane if $c>0$.
Let $f:\SS^{n-1}\times\I \to \Mbar^{n+1}(c)\subset\R^{n+2}$ be the minimal
spherical rotation hypersurface generated by $\gamma$, which is parametrized
as follows
\begin{equation}\label{eq:rotation_hypersurface_1}
       f(t_1,\ldots,t_{n-1},s)=
       (x_{1}(s)\varphi_{1},\ldots,x_{1}(s)\varphi_{n},x_{n+1}(s),x_{n+2}(s))
\end{equation}
if $c\ne{}0$, or
\begin{equation}\label{eq:rotation_hypersurface_2}
       f(t_1,\ldots,t_{n-1},s)=
       (x_{1}(s)\varphi_{1},\ldots,x_{1}(s)\varphi_{n},x_{n+1}(s))
\end{equation}
if $c=0$, where $\varphi(t_1,\ldots,t_{n-1})=(\varphi_{1},\ldots,\varphi_{n})$
is the orthogonal parametrization of the $(n-1)$-unit sphere of the subspace of
$\R^{n+2}$ spanned by $e_1,\ldots,e_{n}$.

Let $\CC=f(\SS^{n-1}\times\I)$ be the minimal rotation
hypersurface immersed in $\Mbar^{n+1}(c)\subset\R^{n+2}$.
According to the computation
in \cite[$\S$3]{dCD83}, we have the first fundamental form of $\CC$:
\begin{equation}\label{eq:first_fundamental_form}
   g_{ij}=\begin{cases}
             \alpha_{ij}x_{1}^{2}(s), & 1\leq{}i,j\leq{}n-1\ ,\\
             0, & i=n,j\ne{}n\ \text{or}\ i\ne{}n, j=n\ ,\\
             1, & i=j=n\ ,
          \end{cases}
\end{equation}
where $\alpha_{ij}=\sum\limits_{k=1}^{n}\dfrac{\partial\varphi_{k}}{\partial{}t_{i}}
\dfrac{\partial\varphi_{k}}{\partial{}t_{j}}$ for $1\leq{}i,j\leq{}n-1$.
According to Proposition 3.2 in \cite{dCD83}, the principal
curvatures of $\CC$ are
\begin{equation}\label{eq:principal_curvatures}
   \lambda_{1}=\cdots=\lambda_{n-1}=
   -\frac{\sqrt{1-cx_{1}^{2}-\dot{x}_{1}^{2}}}{x_{1}}
   \quad\text{and}\quad
   \mu=\frac{\ddot{x}_{1}+cx_{1}}{\sqrt{1-cx_{1}^{2}-\dot{x}_{1}^{2}}}\ ,
\end{equation}
where $\dot{x}_{1}$ and $\ddot{x}_{1}$ are first and second derives of $x_{1}$
on $s$ respectively.
Since $\CC$ is a minimal rotation hypersurface, i.e., a catenoid,
in $\Mbar^{n+1}(c)$, then by \cite[(3.13) and (3.16)]{dCD83} we have
\begin{equation}\label{eq:dot_x}
   \dot{x}_{1}^{2} = 1-cx_{1}^{2}-a^2{}x_{1}^{2-2n}\ ,
\end{equation}
and
\begin{equation}\label{eq:ddot_x}
  \ddot{x}_{1} = -cx_{1}+a^{2}(n-1)x_{1}^{1-2n}\ ,
\end{equation}
where $a>0$ is a constant.
Therefore we have the following identities
\begin{equation}\label{eq:square_norm_of_A}
  |A|^2 = (n-1)\cdot\frac{1-cx_{1}^{2}-\dot{x}_{1}^{2}}{x_{1}^{2}}+
            \frac{(\ddot{x}_{1}+cx_{1})^2}{1-cx_{1}^{2}-\dot{x}_{1}^{2}}
        =a^{2}n(n-1)x_{1}^{-2n}\ ,
\end{equation}
where we use the equations \eqref{eq:dot_x} and \eqref{eq:ddot_x} for the
last equality.

If $\phi=\phi(s)$ is a function on $\CC\subset\Mbar^{n+1}(c)$
depending only on the variable $s$,
then the Laplacian and the square norm of the covariant derivative of $\phi$
with respect to the metric \eqref{eq:first_fundamental_form} on $\CC$
respectively are
\begin{equation}\label{eq:Laplacian}
   \Delta{}\phi=\ddot{\phi}+(n-1)\,\frac{\dot{x}_{1}}{x_{1}}\,\dot{\phi}
\end{equation}
and
\begin{equation}\label{eq:covariant_derivative}
   |\nabla{}\phi|^2=\dot{\phi}^{2}\ .
\end{equation}
Since $|A|=a\sqrt{n(n-1)}\,x_{1}^{-n}$ is function on $\CC$
that depends only on $s$, we have the following equalities
\begin{align*}
   |A|{}\Delta|A|+|A|^4
     & = a^{2}n^{2}(n-1)x_{1}^{-4n}(-2a^{2}+2x_{1}^{2n-2}-cx_{1}^{2n})\\
     & = \frac{2}{n}\,|\nabla|A|{}|^2+nc|A|^2 \ ,
\end{align*}
where we use the equations \eqref{eq:dot_x} and \eqref{eq:ddot_x} again.
\end{proof}

\section{Proof of Theorem \ref{thm:catenoid charaterization}}
\label{sec:Proof_Main_Theorem}

\begin{maintheorem}
Let $\Mbar^{n+1}(c)$ be the space form of dimension $n+1$, where
$n\geq{}3$. Suppose that $\Sigma^n$ is a
non totally geodesic complete minimal hypersurface immersed in $\Mbar^{n+1}(c)$.
If the Simons' equation \eqref{eq:Simons equation} holds as an equation at
all nonvanishing points of $|A|$ in $\Sigma^n$,
then $\Sigma^n\subset\Mbar^{n+1}(c)$ is either
\begin{enumerate}
  \item a catenoid if $c\leq{}0$, or
  \item a Clifford minimal hypersurface or a compact Ostuki
        minimal hypersurface if $c>0$.
\end{enumerate}
\end{maintheorem}

\begin{proof}
By the assumption $\Sigma^{n}$ is not totally geodesic, so $|A|$ is a nonnegative
continuous function which does not vanish identically.
Let $p$ be a point such that $|A|(p) > 0$, there exists an open neighborhood
$U$ of $p$ such that $|A| > 0$ in $U$.

We claim that $|\nabla|A||\not\equiv{}0$ in $U$
unless $\Sigma^{n}$ is a Clifford minimal hypersurface in the case when $c>0$.
We need deal with two cases:

\underline{\emph{Case one}}: $c\leq{}0$. Assume $|\nabla|A||\equiv{}0$ in $U$.
Since $\Sigma^{n}$ satisfies the Simons' equation \eqref{eq:Simons equation},
and $|A|$ is a positive constant in $U$,
we have $0<|A|^2=nc\leq{}0$ in the open set $U$, which is a contradiction.

\underline{\emph{Case two}}: $c>0$. In this case, if $|\nabla|A||\equiv{}0$ in $U$,
then $|A|$ is a nonzero constant in $U$, so $|A|^2=nc$ in $U$ according to
\eqref{eq:Simons equation} and the assumption $|A|\ne{}0$ in $U$,
then $\Sigma^{n}$ is a Clifford minimal hypersurface according to
\cite{CdCK70,Law69}.
In this case, we have $|\nabla|A||\not\equiv{}0$ in $U$ if $\Sigma^{n}$
is not a Clifford minimal hypersurface.

Therefore in each case, there is a point in $U$ such that
$|\nabla|A||\ne{}0$ if $\Sigma^{n}$ is not a Clifford minimal hypersurface in the case
when $c>0$.

By shrinking $U$, we may assume that $|A| > 0$ and $|\nabla|A|| > 0$ in $U$
if $\Sigma^{n}$ is not a Clifford minimal hypersurface in the case
when $c>0$. By \eqref{eq:Simons identity} and the fact that
$\Sigma^{n}$ satisfies \eqref{eq:Simons equation} in $U$, we conclude
that $E \equiv{} 0$ in $U$.
According to the argument in \cite[pp. 3457--3458]{TZ09}, the eigenvalues
of $A=(h_{ij})_{n\times{}n}$ are $\lambda$ with multiplicity $n-1$ and
$\mu=-(n-1)\lambda$ with $\lambda>0$ since $|A|>0$. According to
\cite[Theorem 5]{Ots70} and \cite[Corollary 4.4]{dCD83},
$U$ is part of a catenoid $\CC$ in $\Mbar^{n+1}(c)$.

According to the maximal principle of minimal submanifolds,
$\Sigma^n$ is part of a catenoid $\CC$ in
$\Mbar^{n+1}(c)$. We claim that $\Sigma^{n}$ must coincide with the 
catenoid $\CC$. Let $f:\Sigma^{n}\to\CC$ be the inclusion. There are two cases:

\underline{\emph{Case one}}: $c\leq{}0$.
In this case, each catenoid is a simply connected (since $n\geq{}3$)
complete minimal rotation hypersurface embedded in $\Mbar^{n+1}(c)$.
Since $f$ is a local isometry, the inclusion $f$ is a covering map by
Lemma 8.14 in \cite[p.224]{Spi99}.
Therefore $f$ must be an identity map, i.e., $\Sigma^{n}=\CC$.

\underline{\emph{Case two}}: $c>0$.
In this case, since $\Sigma^n$ is a closed minimal hypersurface
immersed in $\SS^{n+1}(c)$,
we then have $\Sigma^{n}=\CC$ according to Theorem 4 and Theorem 5
in \cite{Ots70}.
\end{proof}

\section{Appendix}\label{sec:appendix}

In the appendix, with the help of Mathematica and Ti\emph{k}Z/PGF,
we shall draw some figures of the generating curves
of three dimensional catenoids in the space form $\Mbar^{4}(c)$.

\begin{figure}[htbp]
  \begin{center}
     \includegraphics[scale=0.7]{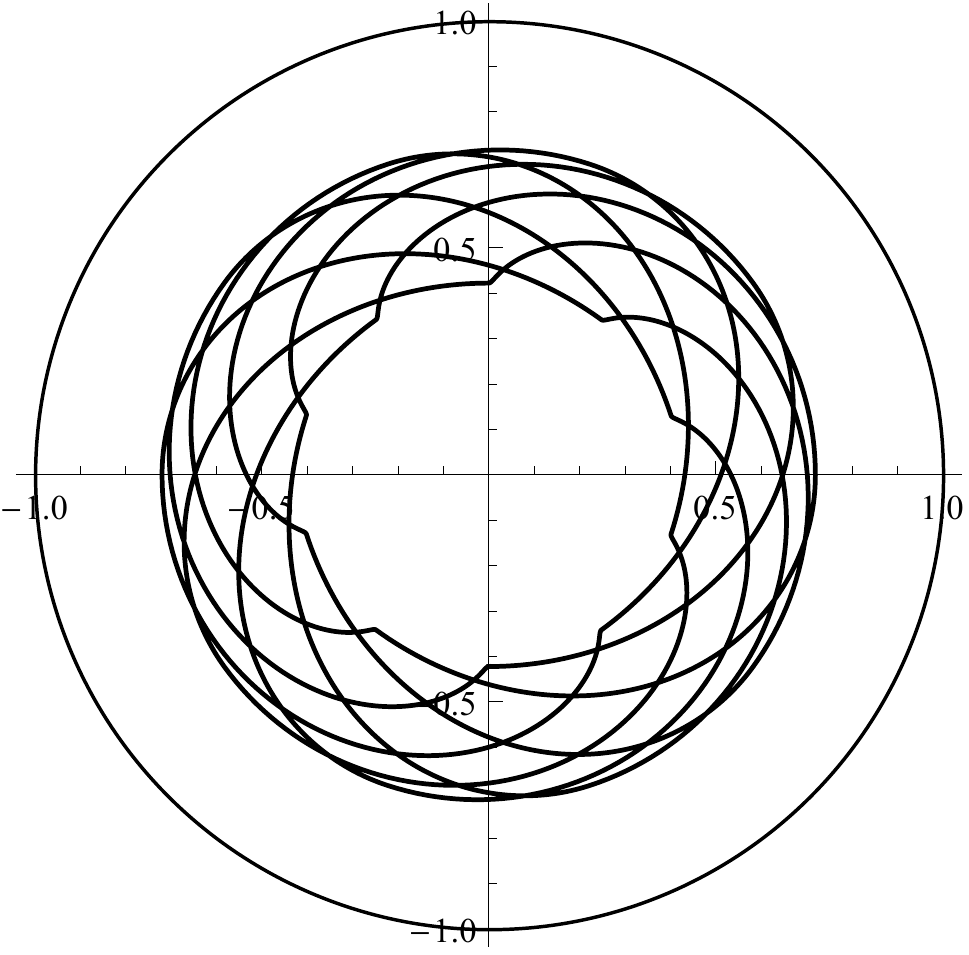}
  \end{center}
  \caption{Generating curve of a compact Otsuki minimal surface in
  the unit sphere $\SS^4$. In this figure, the support function $h$
  has initial conditions $h(0)=a_{0}=0.42231$ and $h'(0)=0$, and
  the period of $h$ is $T=1.4\pi$.}\label{fig:closed generating curve n3}
\end{figure}

\subsection{Catenoids in $\SS^{n+1}$}

In this case we draw the generating curve for a
compact Otsuki minimal hypersurface in $\SS^{4}$.
The function $C(a)$ is given by
\begin{equation}\label{eq:constant C for n=3}
   C(a)=a^{2/3}(1-a^2)^{2/3}\ ,
\end{equation}
where $0<a<1/\sqrt{3}\approx{}0.57735$. Then the following equation
\begin{equation}\label{eq:upper_lower_limits}
   1-x^2-C(a)\left(\frac{1}{x^2}-1\right)^{1/3}=0
\end{equation}
has two solutions $a_0=a$ and $a_1=(-a + \sqrt{4 - 3 a^2})/2\in(1/\sqrt{3},1)$.
Let $a_0=0.42231$, then $a_1=0.71957$, and the period of $h(\theta)$ is
\begin{equation*}
   T=2\int_{a_0}^{a_1}\frac{dx}
     {\sqrt{1-x^2-C(a_0)\left(\dfrac{1}{x^2}-1\right)^{1/3}}}=4.39823\ ,
\end{equation*}
i.e. $T=1.4\pi$. Therefore we have a closed immersed generating curve
as shown in Figure~\ref{fig:closed generating curve n3},
the rotation hypersurface in $\SS^4$ is a compact immersed minimal
hypersurface.

\subsection{Catenoids in $\B^{n+1}$}\label{subsubsec:Hyperbolic_catenoids}
In this case, $c=-1$ and $f(y)=\sinh{}y$, so $f'(y)=\cosh{}y$.
Equation \eqref{eq:generating_curve_function} becomes
\begin{equation}\label{eq:hyperbolic_catenary}
   x(y)=\int_{a}^{y}\frac{1}{\cosh{}t}\cdot
        \frac{dt}{\sqrt{\left(\dfrac{\sinh{}t}{\sinh{}a}\right)^{2n-2}
        \cdot\left(\dfrac{\cosh{}t}{\cosh{}a}\right)^2-1}}\ ,
\end{equation}
where $a\leq{}y<\infty$.

Now let $n=3$, and let $a=0.2$ and $a=1$ respectively in
\eqref{eq:hyperbolic_catenary}, then we have two generating
curves as shown in Figure \ref{fig:generating curves in hyperbolic space}.

\begin{figure}[htbp]
  \begin{center}
    \begin{minipage}{.5\textwidth}
    \centering
    \includegraphics[scale=0.85]{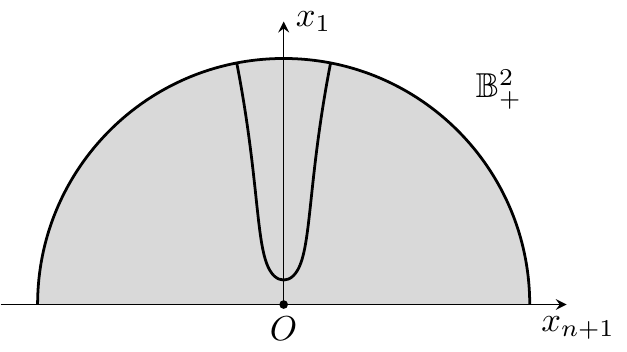}
    \end{minipage}%
    \begin{minipage}{.5\textwidth}
    \centering
    \includegraphics[scale=0.85]{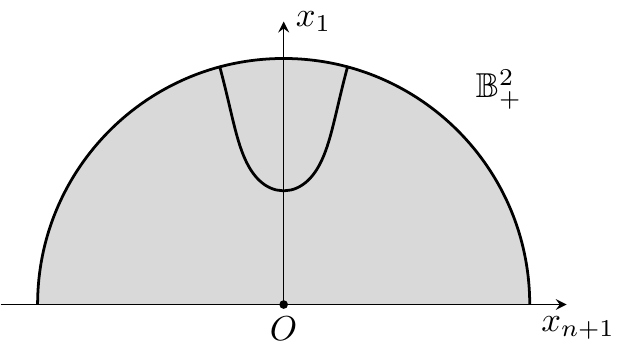}
    \end{minipage}%
  \end{center}
  \caption{Two generating curves for the catenoids in the hyperbolic space
  $\B^4$. In these figures, $a=0.2$ and $a=1$ respectively. The rotation
  axis is the $x_{n+1}$-axis.}\label{fig:generating curves in hyperbolic space}
\end{figure}

\subsection{Catenoids in $\R^{n+1}$}\label{subsubsec:Euclidean_catenoids}
In this case, $c=0$ and $f(y)=y$, so $f'(y)=1$.
Equation \eqref{eq:generating_curve_function} becomes
\begin{equation}\label{eq:Euclidean_generating_curve}
   x(y)=\int_{a}^{y}
        \frac{dt}{\sqrt{(t/a)^{2n-2}-1}}\ ,
\end{equation}
where $a\leq{}y<\infty$. It's easy to see that the integral
\begin{equation}\label{eq:Euclidean_catenary}
   x(a,\infty)
        =\int_{a}^{\infty}
         \frac{dt}{\sqrt{(t/a)^{2n-2}-1}}
        =a\int_{1}^{\infty}\frac{dt}{\sqrt{t^{2n-2}-1}}
\end{equation}
is always finite if $n\geq{}3$, where $a>0$ is a constant.
Actually if $n\geq{}3$, then
\begin{equation*}
   \int_{1}^{\infty}\frac{dt}{\sqrt{t^{2n-2}-1}}\leq
   \int_{1}^{\infty}\frac{dt}{\sqrt{t^{4}-1}}<
   \int_{1}^{\infty}\frac{dt}{t\sqrt{t^{2}-1}}=
   \int_{0}^{\infty}\frac{dx}{\cosh{}x}=\frac{\pi}{2}\ ,
\end{equation*}
where we use the substitution $t=\cosh{}x$.

Now let $n=3$, and let $a=0.5$ and $a=1$ respectively in
\eqref{eq:Euclidean_generating_curve}, then we have two generating
curves as shown in Figure \ref{fig:generating curves in Euclidean space}.

\begin{figure}[htbp]
  \begin{center}
    \begin{minipage}{.5\textwidth}
    \centering
    \includegraphics[scale=0.75]{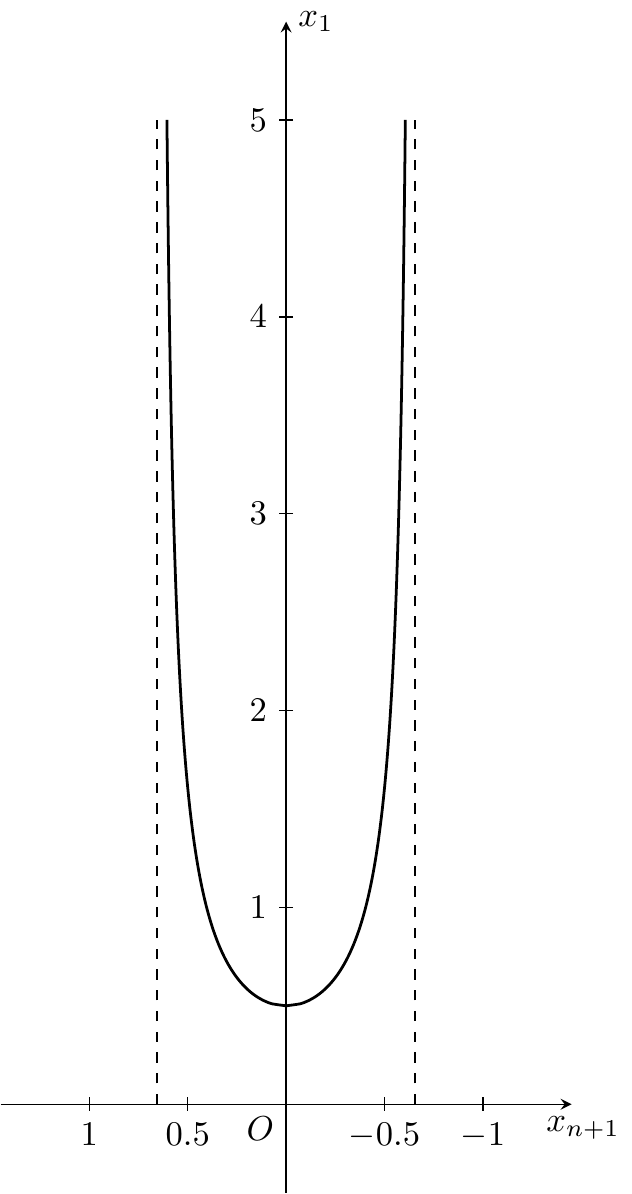}
    \end{minipage}%
    \begin{minipage}{.5\textwidth}
    \centering
    \includegraphics[scale=0.75]{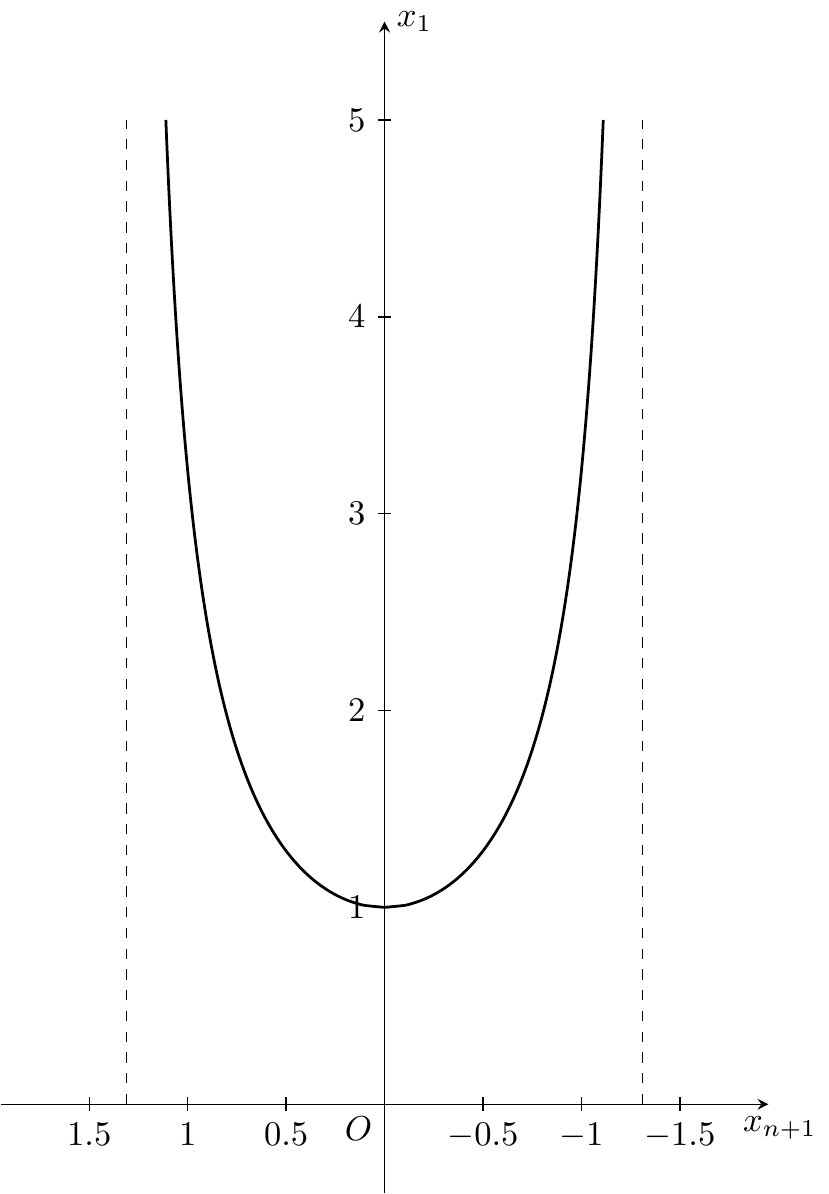}
    \end{minipage}%
  \end{center}
  \caption{Two generating curves for the catenoids in the Euclidean space
  $\R^4$. In these figures, $a=0.5$ and $a=1$ respectively. The rotation
  axis is the $x_{n+1}$-axis.}\label{fig:generating curves in Euclidean space}
\end{figure}



\providecommand{\bysame}{\leavevmode\hbox to3em{\hrulefill}\thinspace}
\providecommand{\MR}{\relax\ifhmode\unskip\space\fi MR }
\providecommand{\MRhref}[2]{%
  \href{http://www.ams.org/mathscinet-getitem?mr=#1}{#2}
}
\providecommand{\href}[2]{#2}

\end{document}